\documentclass[12pt,amscd]{amsart}
\usepackage[all]{xy}
\usepackage{graphicx}
\usepackage{amsmath,amsxtra,amssymb,latexsym, amscd,amsthm}
\usepackage{indentfirst}
\usepackage[mathscr]{eucal}
\usepackage[pagebackref=true]{hyperref}

\newtheorem{thm}{Theorem}[section]

\newtheorem{lem}[thm]{Lemma}

\theoremstyle{definition}
\newtheorem{defn}[thm]{Definition}
\newtheorem{exm}[thm]{Example}

\newtheorem{rem}[thm]{Remark}
\newtheorem{conj}[thm]{Conjecture}

\numberwithin{equation}{section}

\DeclareMathOperator{\pd}{pd}

\DeclareMathOperator{\lcm}{lcm}

\DeclareMathOperator{\supp}{supp}
\DeclareMathOperator{\Tor}{Tor}
\DeclareMathOperator{\level}{level}
\DeclareMathOperator{\reg}{reg}

\def\k {\mathrm{k}}

\begin{document}

\title{Projective dimension and regularity of 3-path ideals of unicyclic graphs}

\author{Nguyen Thu Hang}
\address{Thai Nguyen University of Sciences, Tan Thinh Ward, Thai Nguyen City, Thai Nguyen, Vietnam}
\email{hangnt@tnus.edu.vn}

\author{Thanh Vu}
\address{Institute of Mathematics, VAST, 18 Hoang Quoc Viet, Hanoi, Vietnam}
\email{vuqthanh@gmail.com}


\date{}

\commby{}
\maketitle
\begin{abstract}
    We compute the projective dimension and regularity of $3$-path ideals of arbitrary graphs having at most one cycle.
\end{abstract}

\maketitle

\section{Introduction}
\label{sect_intro}
Let $S = \k[v \mid v \in V]$ be a standard graded polynomial ring over an arbitrary field $\k$ on a set of indeterminates $V$. Let $G$ be a graph on the vertex set $V(G) = V$ and edge set $E(G) \subset V(G) \times V(G)$. Conca and De Negri \cite{CD} introduced the $t$-path ideal of $G$. Note that $I_2(G)$ is the usual edge ideal of $G$. When $G$ is a tree, they showed that the generators of $I_t(G)$ form an $M$-sequence. Jacques \cite{J} described the Betti numbers of the edge ideals of cycles and gave a recursive formula for the Betti numbers of edge ideals of forests. For the $t$-path ideals, Erey \cite{E} gave a formula for the Betti numbers of path ideals of line and star graphs, Alilooee and Faridi \cite{AF2} gave a formula for the Betti numbers of path ideals of line and cycle graphs. Bouchat, Ha, and O'Keefe \cite{BHO}, He and Van Tuyl \cite{HV}, Kiani and Madani \cite{KM}, and Bouchat and Brown \cite{BB} obtained some partial results for the Betti numbers of directed rooted trees. We note that the path ideals of a directed rooted tree are generally smaller than those of the underlying undirected trees associated with it. Banerjee \cite{B} proved that when $G$ is gap-free, one has $\reg (S/I_3(G)) \le \max \{\reg (S/I_2(G)),2\}$. 

Note that if $G = G_1 \cup G_2$ where $G_1$ and $G_2$ are disjoint, then $I_t(G) = I_t(G_1) + I_t(G_2)$ is the sum of ideals in a different set of variables. By \cite{NV2}, we can deduce homological invariants of (powers) of $I_t(G)$ from those of $I_t(G_1)$ and $I_t(G_2)$. Thus, we may assume that $G$ is connected. Even when $G$ is a tree, we do not have a complete answer for the projective dimension and regularity of $I_t(G)$. 

Recently, Kumar and Sarkar \cite{KS} showed that $\reg (S/I_3(G)) = 2 \nu_3(G)$ when $G$ is a tree and when $G$ has a unique cycle then
$$2 \nu_3(G) \le \reg \left ( S/I_3(G) \right ) \le 2 \nu_3(G) + 2$$ 
where $\nu_3(G)$ is the $3$-path induced matching number of $G$. Our motivation in this work is to give a formula for the regularity of $I_3(G)$ when $G$ has a unique cycle in terms of the $3$-path induced matching number of $G$.

We now describe our main results. All graphs in this work are undirected simple graphs. We assume that $G$ has at most one cycle $C$ of length $m$ on the vertices $v_1, \ldots, v_m$. When $G$ is a tree, we set $m = 1$. We first establish certain Betti splittings for the $3$-path ideals of $G$. We then deduce efficient recursive formulae for the projective dimension and regularity of $I_3(G)$, see Theorems \ref{thm_pd_1}, \ref{thm_pd_2}, \ref{thm_reg_1}, \ref{thm_reg_2} for more details. Our results are new even when $G$ is a tree. From that, we deduce a formula for the regularity of $I_3(G)$ in terms of its $3$-path induced matching number, completing the work of Kumar and Sarkar \cite{KS}. 

Let $\partial(C) = N(C) \setminus C$, the boundary of $C$ in $G$, be the set of neighbors of vertices of $C$ that are not in $C$. Let $\Gamma_G(C)$ be the induced subgraph of $G$ on $V(G) \setminus \partial(C)$. Motivated by a result of Alilooee, Beyarslan, and Selvaraja \cite[Corollary 3.9]{ABS}, we prove:
\begin{thm}\label{thm_reg} Assume that $G$ is a connected simple graph with at most one cycle $C$ of length $m$. Then 
$$\reg (S/I_3(G)) = \begin{cases} 2 \nu_3 (G) + 2, & \text{ if } m > 3, m \equiv  3 \pmod 4 \text{ and } \nu_3(G) = \nu_3(\Gamma_G(C)),  \\
2 \nu_3(G) + 1, & \text{ if } m > 3, m \equiv 2 \pmod 4 \text{ and } \nu_3(G) = \nu_3(\Gamma_G(C)), \\
2 \nu_3(G), & \text{ otherwise}.    
\end{cases}$$    
\end{thm}

We structure the paper as follows. In Section \ref{sec_pre}, we set up the notation, provide some background, and prove some properties of $3$-induced matching numbers of graphs. In Section \ref{sec_splittings}, we establish certain Betti splittings involving $3$-path ideals of graphs having at most one cycle. In Section \ref{sec_projective_dimension}, we compute the projective dimension and regularity of $3$-path ideals of graphs having at most one cycle. We then deduce Theorem \ref{thm_reg}.

\section{Preliminaries}\label{sec_pre}

Throughout this section, we let $S = \k[x_1,\ldots, x_n]$ be a standard graded polynomial ring over an arbitrary field $\k$. 

\subsection{Projective dimension and regularity} Let $M$ be a finitely generated graded $S$-module. The $(i,j)$-graded Betti number of $M$ is defined by 
$$\beta_{i,j}(M) = \dim_\k \Tor_i^S (\k,M)_j.$$
The $i$th Betti number of $M$ is $\beta_i(M) = \dim_\k \Tor_i^S(\k,M)$. We define the projective dimension and regularity of $M$ as follows.
\begin{align*}
    \pd_S (M) &= \sup \{i \mid \beta_i(M) \neq 0\},\\
    \reg_S(M) &= \sup \{ j -i \mid \beta_{i,j} (M) \neq 0\}.
\end{align*}

For simplicity of notation, we often write $\pd (M)$ and $\reg (M)$ instead of $\pd_S(M)$ and $\reg_S(M)$. We have the following well known facts.
\begin{lem}\label{lem_mul_x} Let $x$ be a variable and $I$ a nonzero homogeneous ideal of $S$. Then 
\begin{enumerate}
    \item $\pd (I) = \pd (S/I) -1$,
    \item $\reg (I) = \reg (S/I) +1$,
    \item $\pd (xI) = \pd (I)$,
    \item $\reg (xI) = \reg (I) + 1$.
\end{enumerate}    
\end{lem}

\subsection{Betti splittings} Betti splittings was introduced by Francisco, Ha, and Van Tuyl \cite{FHV} for monomial ideals. It has regained interest recently in some work \cite{CF, CFL}. We recall the definition and the following results about Betti splittings from \cite{NV2}.

\begin{defn}
    Let $P, I, J$ be proper nonzero homogeneous ideals of $S$ such that $P = I + J$. The decomposition $P = I +J$ is called a Betti splitting if for all $i \ge 0$ we have $\beta_i(P) = \beta_i(I) + \beta_i(J) + \beta_{i-1}(I \cap J)$.
\end{defn}

\begin{lem}\label{lem_pd_reg_split} Assume that $P = I + J$ is a Betti splitting of ideals of $S$. Then 
\begin{align*}
    \pd_S(P) &= \max \{ \pd_S(I), \pd_S(J), \pd_S(I \cap J) + 1\},\\
    \reg_S(P) &= \max\{ \reg_S(I), \reg_S(J), \reg_S(I \cap J) - 1\}.    
\end{align*}
\end{lem}
\begin{proof}
    See \cite[Corollary 2.2]{FHV} or \cite[Lemma 3.7]{NV2}.
\end{proof}

\begin{defn} Let $\varphi: M \to N$ be a morphism of finitely generated graded $S$-modules. We say that $\varphi$ is $\Tor$-vanishing if for all $i \ge 0$, we have $\Tor_{i}^S(\k,\varphi) = 0$.    
\end{defn}

\begin{lem}\label{lem_splitting_criterion_1} Let $I, J$ be nonzero homogeneous ideals of $S$ and $P = I+J$. The decomposition $P = I +J$ is a Betti splitting if and only if the inclusion maps $I \cap J \to I$ and $I\cap J \to J$ are $\Tor$-vanishing.    
\end{lem}
\begin{proof}
    See \cite[Lemma 3.5]{NV2}.
\end{proof}
For a monomial ideal $I$ of $S$, we denote by $\partial^*(I)$ the ideal generated by elements of the form $f/x_i$, where $f$ is a minimal monomial generator of $I$ and $x_i$ is a variable dividing $f$. We have 

\begin{lem}\label{lem_splitting_criterion_2} Let $I\subseteq J$ be monomial ideals of $S$. Assume that $\partial^*(I) \subseteq J$. Then the inclusion map $I \to J$ is $\Tor$-vanishing.    
\end{lem}
\begin{proof}
    Applying \cite[Lemma 4.2, Proposition 4.4]{NV2}, we deduce the conclusion.
\end{proof}

\subsection{Graphs and their path ideals}\label{subsection_graphs} 

Let $G$ denote a finite simple graph over the vertex set $V(G) = \{x_1,\ldots,x_n\}$ and the edge set $E(G)$. For a vertex $x\in V(G)$, let the neighborhood of $x$ be the subset $N_G(x)=\{y\in V(G) \mid \{x,y\}\in E(G)\}$. The closed neighborhood of $x$ is $N_G[x]=N_G(x)\cup\{x\}$. A leaf is a vertex having a unique neighbor.

A simple graph $H$ is a subgraph of $G$ if $V(H) \subseteq V(G)$ and $E(H) \subseteq E(G)$. $H$ is an induced subgraph of $G$ if it is a subgraph of $G$ and $E(H) = E(G) \cap V(H) \times V(H)$.

For a subset $U \subset V(G)$, we denote by $N_G(U) = \cup \left ( N_G(x) \mid x \in U \right )$ and $N_G[U] = \cup \left ( N_G[x] \mid x \in U \right )$. When it is clear from the context, we drop the subscript $G$ from the notation $N_G$. The induced subgraph of $G$ on $U$, denoted by $G[U]$, is the graph such that $V(G[U]) = U$ and $E(G[U]) = E(G) \cap U \times U$. The notation $G\setminus U$ denotes the induced subgraph of $G$ on $V(G) \setminus U$.

A cycle $C$ of length $m$ is the graph on $V(C) = \{x_1,\ldots,x_m\}$ whose edges are $\{x_1,x_2\}, \ldots, \{x_{m-1}x_m\} , \{x_m,x_1\}$. 

A graph $G$ is called unicyclic if $G$ has a unique cycle.

A tree is a connected graph without any cycle.

Let $t \ge 2$ be a natural number. A $t$-path of $G$ is a sequence of distinct vertices $x_1, \ldots, x_t$ of $G$ such that $\{x_1,x_2\}$,$\ldots, \{x_{t-1},x_t\}$ are edges of $G$. A $t$-path matching in a graph $G$ is a subgraph consisting of pairwise disjoint $t$-paths. If the subgraph is induced, then $t$-path matching is said to be a $t$-path induced matching of $G$. The $t$-path induced matching number of $G$, denoted by $\nu_t(G)$, is the largest size of a $t$-path induced matching of $G$. 

From the definition, we have
\begin{lem}\label{lem_ind_mat_properties} Let $G$ be a simple graph. We have
\begin{enumerate}
    \item Assume that $H$ is an induced subgraph of $G$. Then $\nu_t(H) \le \nu_t(G)$.
    \item Assume that $G = G_1 \cup G_2$ where $G_1$ and $G_2$ are disjoint. Then $\nu_t(G) = \nu_t(G_1) + \nu_t(G_2)$.
\end{enumerate}    
\end{lem}

The $t$-path ideal of $G$ is defined to be
$$I_t(G)= \left ( x_{i_1} \cdots x_{i_t} \mid x_{i_1}, \ldots,x_{i_t} \text{ is a } t\text{-path of } G \right ) \subseteq S.$$

For simplicity of notation, we often write $x_i \in G$ instead of $x_i \in V(G)$. By abuse of notation, we also call a monomial $x_{i_1} \cdots x_{i_t} \in I_t(G)$ a $t$-path of $G$.

\begin{lem}\label{lem_induced_sub} Let $H$ be an induced subgraph of $G$. Then $\reg (I_t(H)) \le \reg (I_t(G))$ and $\pd (I_t(H)) \le \pd (I_t(G))$. In particular, $\reg (I_t(G)) \ge (t-1) \nu_t(G) + 1$.  
\end{lem}
\begin{proof}
    Since $H$ is an induced subgraph of $G$, $I_t(H)$ is the restriction of $I_t(G)$ on $V(H)$. The conclusion follows from \cite[Corollary 2.5]{OHH}.

    Now let $H$ be a maximum $t$-induced matching of $G$. Then $I_t(H)$ is the complete intersection of $\nu_t(G)$ forms of degree $t$. Hence, $\reg (I_t(H)) = (t-1) \nu_t(G) + 1$. The conclusion follows.
\end{proof}

\subsection{3-path induced matching numbers}\label{sub_ind_mat} In this subsection, we give a recursive formula for computing the $3$-path induced matching numbers of graphs having at most one cycle. Let $G$ be a connected graph having at most one cycle of length $m$ on the vertices $v_1, \ldots, v_m$. When $G$ is a tree, we set $m = 1$. Attached to each $v_i$, $i = 1, \ldots,m$, we have a tree $T_i$, where we may consider $T_i$ as a rooted tree with root $v_i$. Note that $E(T_i)$ might be empty. For each vertex $y \in T_i$, the level of $y$, denoted by $\level(y)$, is the length of the unique path from $v_i$ to $y$. Fix an $i$ such that $E(T_i)$ is non-empty. Let $z_0 \in T_i$ be a leaf of $G$ of the highest level among all leaves in $T_i$. Let $y_0$ be the unique neighbor of $z_0$. 

First, assume that $\level(z_0) \ge 2$. Then $y_0 \notin \{v_1, \ldots, v_m\}$ has a unique neighbor that is not a leaf. We assume that $N(y_0) = \{z_0, z_1, \ldots, z_s\} \cup \{x_0\}$, where $z_1, \ldots, z_s$ are leaves and $x_0$ is not a leaf of $G$. We define the following graphs
\begin{equation}\label{eq_graphs_level_2}
    G_{z_0,1} = G \setminus \{z_0,\ldots,z_s\}, G_{z_0,2} = G \setminus N[y_0], \text{ and } G_{z_0,3} = G \setminus N[\{y_0,x_0\}]. 
\end{equation}

\begin{lem}\label{lem_mat1} With the notations above, we have 
$$\nu_3(G) = \begin{cases} \max \{ \nu_3(G_{z_0,1}), 1 + \nu_3(G_{z_0,2}) \} & \text{ if } s \ge 1,\\
\max \{ \nu_3(G_{z_0,1}), 1 + \nu_3(G_{z_0,3}) \} & \text{ if } s = 0.    
\end{cases}$$
\end{lem}
\begin{proof} By Lemma \ref{lem_ind_mat_properties}, the left-hand side is at least the right-hand side. Let $H$ be an induced $3$-path matching of $G$ of maximum size. If $V(H) \cap \{z_0, \ldots,z_s\} = \emptyset$ then $H$ is an induced $3$-path matching of $G_{z_0,1}$. Thus, we may assume that $z_0 \in V(H)$. If $t = 0$ then the $3$-path must be $z_0y_0x_0$ and we have $\nu_3(G) = 1 + \nu_3(G_{z_0,3})$. If $t > 0$ then the $3$-path is either $z_0 y_0 z_j$ for some $j > 0$ or $z_0 y_0 x_0$. In either cases, $\nu_3(G) \le 1 + \nu_3(G_{z_0,2})$. The conclusion follows.    
\end{proof}

Now assume that $\level(z_0) = 1$. We may assume that $v_1$ is the unique neighbor of $z_0$. We assume that $N(v_1) = \{z_0,z_1,\ldots,z_s\} \cup \{v_2,v_m\}$. Note that $z_1, \ldots, z_s$ are leaves of level $1$. We define the following graphs 
\begin{equation} \label{eq_graphs_level_1}
\begin{split}
G_{z_0,1} &= G \setminus \{z_0,\ldots,z_s\}, G_{z_0,2} = G \setminus N[v_1],\\ 
G_{z_0,3} &= G \setminus N[\{v_1,v_2\}], G_{z_0,4} = G \setminus N[\{v_1,v_m\}].
\end{split}
\end{equation}

\begin{lem}\label{lem_mat2} With the notations above, we have 
$$\nu_3(G) = \begin{cases} \max \{ \nu_3(G_{z_0,1}), 1 + \nu_3(G_{z_0,2}) \} & \text{ if } s \ge 1,\\
\max \{ \nu_3(G_{z_0,1}), 1 + \nu_3(G_{z_0,3}),1 + \nu_3(G_{z_0,4}) \} & \text{ if } s = 0.    
\end{cases}$$
\end{lem}
\begin{proof} The proof is similar to that of Lemma \ref{lem_mat1} with a remark that when $t=0$ and $z_0 \in H$ where $H$ is a maximum $3$-induced matching of $G$ then either $z_0, v_1, v_2$ or $z_0,v_1,v_m \in H$. The conclusion follows.    
\end{proof}

We now compute the $3$-path induced matching numbers of paths and cycles.
\begin{lem}\label{lem_path_3_mat} Let $P_m$ be a path on the vertices $\{1,\ldots,m\}$. Then $\nu_3(P_m) = \lfloor \frac{m+1}{4} \rfloor$.    
\end{lem}
\begin{proof} We prove by induction on $m$. The base case $m \le 3$ is obvious. Now, assume that $m \ge 4$. By Lemma \ref{lem_mat1}, we have 
$$\nu_3(P_m) = \max \left \{ \nu_3(P_{m-1}),1 + \nu_3(P_{m-4}) \right \}.$$
By induction on $m$, the conclusion follows.    
\end{proof}

\begin{lem}\label{lem_cycle_3_mat} Let $C_m$ be a cycle of length $m \ge 4$. Then $\nu_3(C_m) = \lfloor \frac{m}{4} \rfloor$.    
\end{lem}
\begin{proof} We may assume that $v_1,v_2,v_3$ is in a maximum $3$-induced matching of $C_m$. Then $\nu_3(C_m) = 1 + \nu_3(P_{m-5})$. The conclusion follows from Lemma \ref{lem_path_3_mat}.    
\end{proof}

\subsection{Proximal graphs}\label{sub_proximal} Let $G$ be a connected simple graph with a unique cycle $C$ on the vertices $v_1, \ldots, v_m$ where $m \ge 3$. Let $\partial(C) = N_G(C) \setminus C$ be the boundary of $C$ in $G$ and $\Gamma_G(C)$ be the induced subgraph of $G$ on $V(G) \setminus \partial (C)$. 

\begin{defn} We say that $G$ is $t$-proximal if $\nu_t(G) = \nu_t( \Gamma_G(C) ).$    
\end{defn}

We now prove some properties regarding $3$-induced matching numbers of $3$-proximal graphs. We keep the notations as in subsection \ref{sub_ind_mat}.

\begin{lem}\label{lem_3_mat_proximal_1} Let $z_0$ be a leaf of the highest level of $G$ with the unique neighbor $y_0$. Assume that $N(y_0) = \{z_0,z_1,\ldots,z_s\} \cup \{x_0\}$ and $\level(z_0) \ge 2$. Assume that $G$ is not $3$-proximal.
\begin{enumerate}
    \item If $G_{z_0,1}$ is $3$-proximal then $\nu_3(G_{z_0,1}) < \nu_3(G)$.
    \item If $s \ge 1$ and $G_{z_0,2}$ is $3$-proximal then $\nu_3(G_{z_0,2}) + 1 < \nu_3(G)$. 
    \item If $s = 0$ and $G_{z_0,3}$ is $3$-proximal then $\nu_3(G_{z_0,3}) + 1 < \nu_3(G)$. 
\end{enumerate}
\end{lem}
\begin{proof} We prove the lemma by contradiction. 

For (1), assume that $G$ is not $3$-proximal, $G_{z_0,1}$ is $3$-proximal and $\nu_3(G_{z_0,1}) = \nu_3(G)$. Note that $\Gamma_1(C)$ is an induced subgraph of $\Gamma_G(C)$, where $\Gamma_1(C) = \Gamma_{G_{z_0,1}}(C)$. By Lemma \ref{lem_ind_mat_properties}, we have 
$$\nu_3(G) = \nu_3(G_{z_0,1}) = \nu_3(\Gamma_1(C)) \le \nu_3(\Gamma_G(C)) \le \nu_3(G).$$
Hence, $\nu_3(G) = \nu_3(\Gamma_G(C))$. In other words, $G$ is $3$-proximal, a contradiction.

For (2), assume that $s \ge 1$, $G$ is not $3$-proximal, $G_{z_0,2}$ is $3$-proximal and $\nu_3(G_{z_0,2}) + 1 = \nu_3(G)$. Since $G_{z_0,2}$ is $3$-proximal, it contains $C$. Let $\Gamma_2(C) = \Gamma_{G_{z_0,2}} (C)$. Note that $\Gamma_2(C) \cup \{z_0y_0x_0\}$ is an induced subgraph of $\Gamma_G(C)$. By Lemma \ref{lem_ind_mat_properties}, we have 
$$\nu_3(G) = 1 + \nu_3(G_{z_0,2}) = 1 + \nu_3(\Gamma_2(C)) \le \nu_3(\Gamma_G(C)) \le \nu_3(G).$$
Hence, $\nu_3(G) = \nu_3(\Gamma_G(C))$, a contradiction.

One can prove Part (3) similarly to Part (2).    
\end{proof}

\begin{lem}\label{lem_3_mat_proximal_2} Let $z_0$ be a leaf of the highest level of $G$ with $\level(z_0) = 1$. Assume that $v_1$ is the unique neighbor of $z_0$ and $N(v_1) = \{z_0,z_1,\ldots,z_s\} \cup \{v_2,v_m\}$. Assume that $G$ is not $3$-proximal and $G_{z_0,1}$ is $3$-proximal. Then $\nu_3(G_{z_0,1}) < \nu_3(G)$.
\end{lem}
\begin{proof} Assume by contradiction that $\nu_3(G_{z_0,1}) = \nu_3(G)$. Let $\Gamma_1(C) = \Gamma_{G_{z_0,1}}(C)$. Note that $\Gamma_1(C)$ is an induced subgraph of $\Gamma_G(C)$. By Lemma \ref{lem_ind_mat_properties}, we have 
$$\nu_3(G) = \nu_3(G_{z_0,1}) = \nu_3(\Gamma_1(C)) \le \nu_3(\Gamma_G(C)) \le \nu_3(G).$$
Hence, $\nu_3(G) = \nu_3(\Gamma_G(C))$, a contradiction.
\end{proof}

\section{Betti splittings for 3-path ideals of graphs with at most one cycle}\label{sec_splittings}

In this section, we establish certain Betti splittings for the $3$-path ideals of graphs with at most one cycle. First, we prove some general results. 

\begin{lem}\label{lem_splitting_a_leaf} Let $G$ be a simple graph with a leaf $z_0$. Let $y_0$ be the unique neighbor of $z_0$. Then     
$$I_3(G) = z_0y_0 \left ( x \mid x \in N(y_0) \setminus \{z_0\} \right ) + I_3(G_1),$$
where $G_1 = G \setminus \{z_0\}$ is the induced subgraph of $G$ on $V(G) \setminus \{z_0\}$. 
\end{lem}
\begin{proof}
    If a $3$-path of $G$ does not contain $z_0$ then it must be a $3$-path of $G_1$. Now, assume that $p$ is a $3$-path of $G$ that contains $z_0$. Since $z_0$ is a leaf with unique neighbor $y_0$, we must have $p$ is of the form $z_0y_0 x$ with $x \in N(y_0) \setminus \{z_0\}$. The conclusion follows.
\end{proof}

\begin{lem}\label{lem_colon_a_vertex} Let $G$ be a simple graph and $y$ be a vertex of $G$. Then 
    $$I_3(G) : y = (uv \mid u \neq v \in N(y) ) + (uw \mid u \in N(y), w \in N(u) \setminus N[y]) + I_3(G_2),$$
    where $G_2$ is the induced subgraph of $G$ on $V(G) \setminus N[y]$. 
\end{lem}
\begin{proof}
    We denote by $L$ the ideal on the right-hand side. Note that $yuv$ and $yuw$ with $u\neq v \in N(y)$ and $w \in N(u) \setminus N[y]$ are $3$-paths of $G$. Hence, the left-hand side contains the right-hand side. 

    Now, let $p$ be a $3$-path of $G$. If $y \in p$ then $p$ must be of the form $yuv$ and $yuw$ as before, and we have $p/y \in L$. Thus, we may assume that $y \notin p$. If $\supp p \cap N[y] = \emptyset$ then $p$ is a $3$-path of $G_2$ and we are done. Thus, we may assume that $u \in \supp p \cap N(y)$. Now, any $3$-path containing $u$ must also contain a neighbor of $u$. Since we assume that $y \notin p$, hence there exists $w \in N(u) \setminus \{y\}$ such that $uw | p$. In particular, we have $uw \in L$. Hence, $p \in L$. The conclusion follows.
\end{proof}

\begin{lem}\label{lem_splitting_1} Let $G$ be a simple graph with a leaf $z_0$. Let $y_0$ be the unique neighbor of $z_0$. We denote by $G_1 = G \setminus \{z_0\}$ the induced subgraph of $G$ on $V(G) \setminus \{z_0\}$. We define the following two ideals of $S$
\begin{align*}
J(z_0) & = (x \mid x \in N(y_0) \setminus \{z_0\}),   \\
L(z_0) &= (uv \mid u \neq v \in N(y_0) \setminus \{z_0\}) + (uw \mid u \in N(y_0)\setminus \{z_0\}, w \in N(u) \setminus N[y_0]) \\
&+ \sum_{u\in N(y_0) \setminus \{z_0\}} u I_3(G_u),
\end{align*}
where $G_u = G \setminus N[\{u,y_0\}]$ is the induced subgraph of $G$ on $V(G) \setminus N[\{u,y_0\}]$. Then the decomposition $I_3(G) = z_0 y_0J(z_0) + I_3(G_1)$ is a Betti splitting and 
$$z_0 y_0 J(z_0) \cap I_3(G_1) = z_0y_0 L(z_0).$$
\end{lem}
\begin{proof} For simplicity of notation, in the proof, we use $J$ for $J(z_0)$ and $L$ for $L(z_0)$. By Lemma \ref{lem_splitting_a_leaf}, we have $I_3(G) = z_0 y_0 J + I_3(G_1)$. By \cite[Corollary 4.12]{NV1} and the fact that $y_0J$ has a linear free resolution, we deduce that the decomposition $I_3(G) = z_0 y_0 J + I_3(G_1)$ is a Betti splitting. Now, we have
$$z_0 y_0 J \cap I_3(G_1) = z_0 y_0 \left (  J \cap (I_3(G_1):y_0) \right ).$$
By Lemma \ref{lem_colon_a_vertex},
$$I_3(G_1) : y_0 = (uv \mid u \neq v \in N(y_0) ) + (uw \mid u \in N(y_0),w \in N(u) \setminus N[y_0]) + I_3(G_2)$$
where $G_2$ is the induced subgraph of $G$ on $V(G) \setminus N[y_0]$. We need to prove that 
$$J \cap (I_3(G_1) : y_0) = L.$$

From the definition of $L$, $L \subset J$ and $L \subset (I_3(G_1):y_0)$. It remains to prove that for any $f \in J \cap I_3(G_2)$, we have $f \in L$. There exist $u \in J$ and $p \in I_3(G_2)$ such that $f = \lcm (u,p) $. Note that $G_2 = G \setminus N[y_0]$, hence $u \notin V(G_2)$. Thus, $f = up$. If $p \cap N(u) \neq \emptyset$ then $up \in (uv \mid u \neq v \in N(y_0) ) + (uw \mid u \in N(y_0),w \in N(u) \setminus N[y_0])$. If $p \cap N(u) = \emptyset$ then $up \in u I_3(G_u)$. The conclusion follows.     
\end{proof}

We now return to the case $G$ is a connected graph with at most one cycle $C$ on the vertices $v_1,\ldots,v_m$. We use the notations in subsection \ref{sub_ind_mat}. In particular, attached to each $v_i$, we have a tree $T_i$. The level of a vertex $y \in T_i$ is the distance from $y$ to the root $v_i$. Let $z_0 \in T_i$ be a leaf of $G$ of the highest level among all leaves in $T_i$ and $y_0$ its unique neighbor. We also use the following notation. For a nonzero monomial $f \in S$, the support of $f$, denoted by $\supp f$, is the set of variables $x_i$ such that $x_i$ divides $f$. For a monomial ideal $I$ of $S$, we define the support of $I$ by
$$\supp I = \bigcup \left ( \supp f \mid f \text{ is a minimal monomial generator of } I \right ).$$

First, assume that $\level(z_0) \ge 2$. Let $N(y_0) = \{z_0,z_1, \ldots,z_s\} \cup \{x_0\}$ where $z_1, \ldots, z_s$ are leaves of $G$ having the same level as $z_0$ and $x_0$ is the unique neighbor of $y_0$ with smaller level. Recall that $G_{z_0,1} = G \setminus \{z_0,\ldots,z_s\}$, $G_{z_0,2} = G\setminus N[y_0]$ and $G_{z_0,3} = G \setminus N[\{y_0,x_0\}]$. We define the following ideals of $S$
\begin{equation}\label{eq_level_2_0}
    \begin{split}
        U & = (z_i z_j \mid 1 \le i < j \le s) + (z_1,\ldots,z_s) I_3(G_{z_0,2}) \\
        V & = x_0 ( z_1,\ldots,z_s,u \in N(x_0) \setminus \{y_0\} ) + x_0 I_3(G_{z_0,3}).
    \end{split}
\end{equation}
Since $\level(z_0) \ge 2$, $x_0,y_0,z_0$ belong to a tree attached to some vertex $v_j$. Hence, $N(x_0) \setminus N[y_0] = N(x_0) \setminus \{y_0\}$. By Lemma \ref{lem_splitting_1}, we get $L(z_0) = U + V$. Note that, if $s = 0$ then $U = (0)$.

\begin{lem}\label{lem_splitting_level_2} Assume that $s \ge 1$. Then the decomposition $L(z_0) = U +V$ is a Betti splitting and 
    $$U  \cap V = x_0 U.$$    
\end{lem}
\begin{proof} By Eq. \eqref{eq_level_2_0}, we have $L(z_0) = U + V$. Since $x_0 \notin \supp U$, we have 
    $$U \cap V = x_0 \left ( U \cap (z_1,\ldots,z_s, u \mid u \in N(x_0) \setminus \{y_0\} )  + I_3(G_3) \right ) = z_0 U.$$
    
Since $x_0 \notin \supp U$, the inclusion map $x_0 U \to U$ is $\Tor$-vanishing. By Lemma \ref{lem_splitting_criterion_1}, it remains to prove that the map $U \to W$ is $\Tor$-vanishing, where $W = (z_1, \ldots,z_s,u \mid u \in N(x_0) \setminus \{ y_0 \} ) + I_3(G_3)$. By Lemma \ref{lem_splitting_criterion_2}, it suffices to prove that for any minimal generator $f$ of $U$ and any variable $x$ that divides $f$, we have $f/x \in W$. By the definition of $U$, it suffices to prove that $I_3(G_2) \subseteq W$. Indeed, if $p \in I_3(G_2)$ that is not in $I_3(G_3)$ then $p$ must contain a variable $u \in N(x_0) \setminus \{y_0\}$. Hence, $p \in W$. The conclusion follows.
\end{proof}

We now consider the case $\level(z_0) = 1$. We may assume that $v_1$ is the unique neighbor of $z_0$ and that $N(v_1) = \{z_0,z_1,\ldots,z_s\} \cup \{ v_2,v_m\}$. Recall that $G_{z_0,1} = G \setminus \{z_0,\ldots,z_s\}$, $G_{z_0,2} = G\setminus N[v_1]$, $G_{z_0,3} = G \setminus N[\{v_1,v_2\}]$, and $G_{z_0,4} = G \setminus N[\{v_1,v_m\}]$. We define the following ideals  

\begin{equation}
    \begin{split}  
    X &= (z_i z_j \mid 1 \le i < j \le s) + (z_1,\ldots,z_s) I_3(G_{z_0,2})\\
    Y &= v_2 \left ( z_1,\ldots,z_s, u \mid u \in N(v_2) \setminus N[v_1] \right ) + v_2 I_3(G_{z_0,3})\\
    Z   &= v_m \left ( z_1,\ldots,z_s,v_2, w \mid w \in N(v_m) \setminus N[v_1] \right ) + v_m I_3(G_{z_0,4}).
    \end{split}
\end{equation}
By Lemma \ref{lem_splitting_1}, we have $L(z_0) = X + Y + Z$. Note that, if $s = 0$ then $X = (0)$. 
\begin{lem}\label{lem_splitting_level_1_1} The decomposition $L(z_0) = (X+Y) + Z$ is a Betti splitting and $(X +Y) \cap Z = v_m (X + Y)$.
\end{lem}
\begin{proof}
    Since $v_m \notin \supp (X + Y)$, we have 
    \begin{align*}
        (X + Y) \cap Z & = v_m \left ( (X+Y)  \cap ( (z_1,\ldots,z_s,v_2, w \mid w \in N(v_m) \setminus N[v_1] ) + I_3(G_4) ) \right ) \\
        &= v_m(X +Y).
    \end{align*}
    By Lemma \ref{lem_splitting_criterion_1} and Lemma \ref{lem_splitting_criterion_2}, it suffices to prove that $I_3(G_{z_0,2})  \subseteq (z_1,\ldots,z_s,v_2,w \mid w \in N(v_m) \setminus N[v_1]) + I_3(G_{z_0,4})$. Indeed, if $p \in I_3(G_{z_0,2})$ that is not in $I_3(G_{z_0,4})$ then $p$ must contain a variable $u \in N(v_m) \setminus N[v_1]$. The conclusion follows.
\end{proof}

\begin{lem}\label{lem_splitting_level_1_2} Assume that $s \ge 1$. The decomposition $(X+Y) = X+Y$ is a Betti splitting and $X \cap Y = v_2 X$.
\end{lem}
\begin{proof}
    Since $v_2 \notin \supp X$, we have 
    $$X \cap Y = v_2 \left ( X \cap (( z_1,\ldots,z_s, u \mid u \in N(v_2) \setminus N[v_1]  ) + I_3(G_{z_0,3}) ) \right ) = v_2X.$$
    With an argument similar to the proof of Lemma \ref{lem_splitting_level_1_1}, we deduce that the decomposition $(X+Y) = X + Y$ is a Betti splitting.
\end{proof}

\section{Projective dimension and regularity of 3-path ideals of graphs with at most one cycle}\label{sec_projective_dimension}
Let $G$ be a connected graph having at most one cycle $C$ on the vertices $v_1, \ldots,v_m$. With the Betti splittings established in Section \ref{sec_splittings}, one can deduce a recursive formula for computing the Betti numbers of $3$-path ideals of $G$. Nonetheless, the formula is a bit cumbersome, so we focus on deriving the formulae for the projective dimension and regularity. We then deduce Theorem \ref{thm_reg}.

We first have some preparation lemmas.
\begin{lem}\label{lem_U} Let $z_1, \ldots,z_s$ be variables of $S$ and $P$ a nonzero monomial ideal of $S$ such that $z_i \notin \supp P$ for $i = 1, \ldots,s$. Let $Q_s = (z_i z_j \mid 1 \le i < j \le s) + (z_1,\ldots,z_s) P$. Assume that $s \ge 1$. Then 
\begin{enumerate}
    \item $\pd_S( Q_s ) = \pd_S( P) + s-1$,
    \item $\reg_S( Q_s ) = \reg_S( P) + 1$.
\end{enumerate}    
\end{lem}
\begin{proof} We prove by induction on $s$. The base case $s = 1$ is clear. Assume that the statements hold for $s \ge 1$. Let $B = z_{s+1} (z_1,\ldots,z_s) + z_{s+1} P$. Then we have $Q_{s+1} = Q_s + B$ and $Q_s \cap B = z_{s+1} Q_s$. By Lemma \ref{lem_splitting_criterion_1} and Lemma \ref{lem_splitting_criterion_2}, we deduce that the decomposition $Q_{s+1} = Q_s + B$ is a Betti splitting. By induction and Lemma \ref{lem_pd_reg_split}, we deduce that 
\begin{align*}
    \pd_S ( Q_{s+1}) &= \max \{\pd_S(Q_s) + 1, \pd_S(B)\} = \pd_S(P) + s,\\
    \reg_S(Q_{s+1}) & = \max \{ \reg_S (Q_s), \reg_S(B)\} = \reg_S(P) + 1.
\end{align*}
The conclusion follows.
\end{proof}

\begin{lem}\label{lem_V} Let $z_1,\ldots,z_s$ be variables of $S$ and $P$ a non-zero monomial ideal of $S$ such that $z_i \notin \supp P$ for $i= 1,\ldots,s$. Let $Q = (z_1,\ldots,z_s) + P$. Then
\begin{enumerate}
    \item $\pd_S(Q) = \pd_S(P) + s$,
    \item $\reg_S(Q) = \reg_S(P)$.
\end{enumerate}   
\end{lem}
\begin{proof}
    Let $Z = (z_1,\ldots,z_s)$. The conclusion follows from \cite[Proposition 3.11]{NV2} and the facts that $\pd_S(Z) = s-1$ and $\reg_S(Z) = 1$.
\end{proof}

We recall the following notations in Section \ref{sec_pre}. Let $z_0$ be a leaf of $G$ of the highest level and $y_0$ its unique neighbor. First, assume that $\level(z_0) \ge 2$. Let $N(y_0) = \{z_0, z_1, \ldots,z_s\} \cup \{x_0\}$, where $z_1, \ldots,z_s$ are leaves of $G$ having the same level as $z_0$ and $x_0$ is the unique neighbor of $y_0$ of smaller level. We denote by $G_{z_0,1} = G \setminus \{z_0,z_1,\ldots,z_s\}$, $G_{z_0,2} = G \setminus N[y_0]$ and $G_{z_0,3} = G \setminus N[\{y_0,x_0\}]$.

\begin{thm}\label{thm_pd_1} Assume that $\level(z_0) \ge 2$. With the notations above, we have  
$$\pd (I_3(G)) = \begin{cases} \max \{ p_1, p_2 + s+1, p_3 + |N(x_0)| + s\}, & \text{ if } s > 0 \\
\max \{ p_1, p_3 + |N(x_0)|\}, & \text{ if } s = 0.    
\end{cases}$$    
where $p_1 = \pd (I_3(G_{z_0,1})), p_2 = \pd(I_3(G_{z_0,2}))$, and $p_3 = \pd (I_3(G_{z_0,3}) )$.
\end{thm}
\begin{proof} For each $i=0, \ldots, s+1$, let $G_i = G \setminus \{z_0, z_1, \ldots,z_{i-1}\}$. We have $G_0 = G$ and $G_{s+1} = G_{z_0,1}$. For each $i = 0, \ldots, s$, we define the following ideals 
\begin{align*}
    U_i & = (z_j z_k \mid i < j < k \le s) + (z_{i+1},\ldots,z_s) I_3(G_{z_0,2}) \\
    V_i &= x_0(z_{i+1},\ldots,z_s, u \in N(x_0) \setminus \{y_0\}) + x_0 I_3(G_{z_0,3})\\
    L_i &= U_i + V_i.
\end{align*}

    By Lemma \ref{lem_pd_reg_split} and Lemma \ref{lem_splitting_1}, we have 
\begin{equation}
    \pd (I_3(G_i)) = \max \{ s -i, \pd (I_3(G_{i+1})), \pd (L_i) + 1\}.
\end{equation}
Note that $U_s =(0)$. By Lemma \ref{lem_pd_reg_split} and Lemma \ref{lem_splitting_level_2}, we have 
\begin{equation}
    \pd(L_i) = \begin{cases}
        \max \{ \pd(U_i) + 1, \pd (V_i) \}, & \text{ if } i < s,\\
        \pd (V_i), & \text{ if } i = s.
    \end{cases}
\end{equation}

By Lemma \ref{lem_U}, $\pd(U_i) = \pd (I_3(G_{z_0,2})) + s-i-1$ for $i < s$. By Lemma \ref{lem_V}, $\pd(V_i) = \pd (I_3(G_{z_0,3})) + |N(x_0)| + s-i-1$. Hence, 
\begin{equation}
    \pd(L_i) = \begin{cases}
        \max \{ p_2 + s-i, p_3 + |N(x_0)| + s -i -1 \}, & \text{ if } i < s,\\
        p_3 + |N(x_0)| - 1, & \text{ if } i = s.
    \end{cases}
\end{equation}
The conclusion follows.
\end{proof}

\begin{thm}\label{thm_reg_1} Assume that $\level(z_0) \ge 2$. With the notations above, we have  
$$\reg (I_3(G)) = \begin{cases}
    \max \{ \reg (I_3(G_{z_0,1})), \reg (I_3(G_{z_0,2}) ) +2\} & \text{ if } s \ge 1,\\
    \max \{ \reg (I_3(G_{z_0,1})), \reg (I_3(G_{z_0,3}) ) +2\} & \text{ if } s = 0.
\end{cases}$$
\end{thm}
\begin{proof} We use the notations as in the proof of Theorem \ref{thm_pd_1}.  By Lemma \ref{lem_pd_reg_split} and Lemma \ref{lem_splitting_1}, we have 
\begin{equation}
    \reg (I_3(G_i)) = \max \{ \reg (I_3(G_{i+1})), \reg (L_i) + 1\}.
\end{equation}
By Lemma \ref{lem_pd_reg_split} and Lemma \ref{lem_splitting_level_2}, we have 
\begin{equation}
    \reg(L_i) = \max \{ \reg(U_i), \reg (V_i) \}.
\end{equation}
Note that $U_s$ is the zero ideal. By Lemma \ref{lem_U}, $\reg(U_i) = \reg (I_3(G_{z_0,2})) + 1$ if $0 \le i < s$. By Lemma \ref{lem_V}, $\reg(V_i) = \reg (I_3(G_{z_0,3})) + 1$ for all $i = 0, \ldots,s$. Since $G_{z_0,3}$ is an induced subgraph of $G_{z_0,2}$, we have that $\reg (I_3(G_{z_0,3})) \le \reg (I_3(G_{z_0,2}))$. The conclusion follows.
\end{proof}

Now, assume that $\level(z_0) = 1$ and $v_1$ is the unique neighbor of $z_0$. Assume that $N(v_1) = \{z_0,z_1,\ldots,z_s\} \cup \{v_2,v_m\}$. We denote by $G_{z_0,1} = G\setminus \{z_0,z_1,\ldots,z_s\}$, $G_{z_0,2} = G \setminus N[v_1]$, $G_{z_0,3} = G \setminus N[\{v_1,v_2\}]$, and $G_{z_0,4} = G \setminus N[\{v_1,v_m\}]$. 

\begin{thm}\label{thm_pd_2} Assume that $\level(z_0) = 1$. With the notations above, we have  
$$\pd (I_3(G)) = \begin{cases} \max \{ p_1, p_2 + s + 1, p_3 + n_2 + s + 1,p_4 + n_m + s + 1\}, & \text{ if } s > 0 \\
\max \{ p_1, p_3 + n_2 + 1, p_4 + n_m + 1\}, & \text{ if } s = 0.    
\end{cases}$$ 
where $p_1 = \pd(I_3(G_{z_0,1}) ), p_2 = \pd(I_3(G_{z_0,2}))$, $p_3 = \pd (I_3(G_{z_0,3}))$, $p_4 = \pd (I_3(G_{z_0,4}))$, and $n_2 = |N(v_2) \setminus N[v_1]|$ and $n_m = |N(v_m) \setminus N[v_1]|$.
\end{thm}
\begin{proof} For each $i = 0, \ldots,s+1$, let $G_i = G \setminus \{z_0, z_1,\ldots,z_{i-1}\}$. We have $G_0 = G$ and $G_{s+1} = G_{z_0,1}$. For each $i = 0, \ldots,s$, we define the following ideals 
\begin{align*}
    X_i & = (z_j z_k \mid i < j < k \le s) + (z_{i+1},\ldots,z_s) I_3(G_{z_0,2}) \\
    Y_i &= v_2(z_{i+1},\ldots,z_s, u \in N(v_2) \setminus N[v_1] ) + v_2 I_3(G_{z_0,3})\\
    Z_i &= v_m(z_{i+1},\ldots,z_s, v_2,u \in N(v_m) \setminus N[v_1] ) + v_m I_3(G_{z_0,4})\\
    L_i &= X_i + Y_i + Z_i.
\end{align*}
By Lemma \ref{lem_pd_reg_split} and Lemma \ref{lem_splitting_1}, we have 
$$\pd_S ( I_3(G_i) )  = \max \{ s - i, \pd_S (I_3(G_{i+1} )), \pd_S( L_i) + 1 \}.$$
By Lemma \ref{lem_pd_reg_split}, Lemma \ref{lem_splitting_level_1_1} and Lemma \ref{lem_splitting_level_1_2}, we have 
\begin{align*}
    \pd_S( L_i)  &= \max \{ \pd (X_i + Y_i) + 1, \pd (Z_i)\} \\
    &= \begin{cases} 
\max \{ \pd(X_i) + 2, \pd (Y_i) + 1, \pd (Z_i) \} & \text{ if } i < s,\\
\max \{ \pd (Y_i) + 1, \pd (Z_i) \} & \text{ if } i = s.\end{cases}
\end{align*}
Note that $X_s = (0)$. By Lemma \ref{lem_U} and Lemma \ref{lem_V}, we deduce that 
\begin{align*}
    \pd (X_i) & = \pd_S (I_3(G_{z_0,2})) +s - i -1 \text { if } i < s,\\
    \pd(Y_i) &= \pd (I_3(G_{z_0,3}) )+ |N(v_2) \setminus N[v_1] | + s - i -1,\\
    \pd (Z_i) &= \pd_S(I_3(G_{z_0,4})) + |N(v_m) \setminus N[v_1] | + s-i.
\end{align*}
The conclusion follows.  
\end{proof}

\begin{thm}\label{thm_reg_2} Assume that $\level(z_0) = 1$. Then 
$$\reg (I_3(G)) = \begin{cases}
    \max \{ \reg (I_3(G_{z_0,1})), \reg (I_3(G_{z_0,2}) ) +2\} & \text{ if } s \ge 1,\\
    \max \{ \reg (I_3(G_{z_0,1})), \reg (I_3(G_{z_0,3}) ) +2, \reg (I_3(G_{z_0,4})) + 2\} & \text{ if } s = 0.
\end{cases}$$
\end{thm}
\begin{proof} We use the notations as in the proof of Theorem \ref{thm_pd_2}. By Lemma \ref{lem_pd_reg_split} and Lemma \ref{lem_splitting_1}, we have 
$$\reg ( I_3(G_i) )  = \max \{ \reg (I_3(G_{i+1} )), \reg ( L_i) + 1 \}.$$
By Lemma \ref{lem_pd_reg_split}, Lemma \ref{lem_splitting_level_1_1}, and Lemma \ref{lem_splitting_level_1_2}, we have 

\begin{align*}
    \reg ( L_i)  &= \max \{ \reg (X_i + Y_i), \reg (Z_i)\} \\
    &= \begin{cases} 
\max \{ \reg(X_i), \reg (Y_i), \reg (Z_i) \} & \text{ if } i < s,\\
\max \{ \reg (Y_i), \reg (Z_i) \} & \text{ if } i = s.\end{cases}
\end{align*}
By Lemma \ref{lem_U}, we have $\reg (X_i) = 1 + \reg (I_3(G_{z_0,2}))$ if $i < s$. By Lemma \ref{lem_V}, $\reg (Y_i) = 1 + \reg (I_3(G_{z_0,3}))$ and $\reg (Z_i) = 1 + \reg (I_3(G_{z_0,4}))$. Since $G_{z_0,3}$ and $G_{z_0,4}$ are induced subgraphs of $G_{z_0,2}$, the conclusion follows.
\end{proof}

We now give an example to illustrate our results.

\begin{exm}\label{ex1} Let $G$ be a graph on the vertex set $V(G) = \{v_1,\ldots,v_6\} \cup \{x_1,\ldots,x_6\} \cup \{y_1,\ldots,y_6\} \cup \{z_1,\ldots,z_5\}$ such that its edge ideal is 
\begin{align*}
    I(G) &= (v_1v_2,\ldots,v_5v_6,v_1v_6) + v_1 (x_1,\ldots,x_4) + v_2(x_5,x_6) \\
    &+ x_1(y_1,\ldots,y_4) + x_5(y_5,y_6) + y_1(z_1,\ldots,z_4) + y_6(z_5).
\end{align*} 
Then $z_1$ is a leaf of the highest level of $G$, and we have $G_{z_1,1},G_{z_1,2}$, and $G_{z_1,3}$ are graphs whose edge ideals are 
\begin{align*}
    I(G_{z_1,1}) &= (v_1v_2,\ldots,v_5v_6,v_1v_6) + v_1 (x_1,\ldots,x_4) + v_2(x_5,x_6) \\
    &+ x_1(y_1,\ldots,y_4) + x_5(y_5,y_6) + y_6(z_5),\\
    I(G_{z_1,2}) &= (v_1v_2,\ldots,v_5v_6,v_1v_6) + v_1 (x_2,\ldots,x_4) + v_2(x_5,x_6) \\
    &+ x_5(y_5,y_6) + y_6(z_5),\\
    I(G_{z_1,3}) &= (v_2v_3,\ldots,v_5v_6) + v_2(x_5,x_6) + x_5(y_5,y_6) + y_6(z_5).
\end{align*} 
Note that $G_{z_1,3}$ is a tree and one can check easily that $\pd(I(G_{z_1,3})) = 4$. By Theorem \ref{thm_pd_1} and Theorem \ref{thm_reg_1}, we deduce that 
\begin{align*}
    \pd(I_3(G)) &= \max \{ \pd (I_3(G_{z_1,1}) ), \pd (I_3(G_{z_1,2})) + 4, \pd (I_3(G_{z_1,3})) + 8\},\\
    \reg (I_3(G)) &= \max \{ \reg (I_3(G_{z_1,1})), \reg (I_3(G_{z_1,2})) + 2\}.
\end{align*}
Let $H = G_{z_1,1}$ for simplicity of notations. Then $z_5$ is a leaf of the highest level in $H$ and we have 
\begin{align*}
    I(H_{z_5,1}) &= (v_1v_2,\ldots,v_5v_6,v_1v_6) + v_1 (x_1,\ldots,x_4) + v_2(x_5,x_6) \\
    &+ x_1(y_1,\ldots,y_4) + x_5(y_5,y_6),\\
    I(H_{z_5,3}) &= (v_3v_4,\ldots,v_5v_6,v_1v_6) + v_1(x_1,\ldots,x_4) + x_1(y_1,\ldots,y_4).
\end{align*} 
By Theorem \ref{thm_pd_1} and Theorem \ref{thm_reg_1}, we deduce that 
\begin{align*}
    \pd(I_3(H)) &= \max \{ \pd (I_3(H_{z_5,1}) ), \pd (I_3(H_{z_5,3})) + 3\},\\
    \reg (I_3(H)) &= \max \{ \reg (I_3(H_{z_5,1})), \reg (I_3(H_{z_5,3})) + 2\}.
\end{align*}
The graphs $H_{z_5,1}$ and $H_{z_5,3}$ are relatively small, and we can compute the projective dimension and regularity of their $3$-path ideals using Macaulay2 \cite{M2}. We then deduce that $\pd(I_3(H)) = 11$ and $\reg (I_3(H)) = 7$. We can also apply Theorem \ref{thm_pd_1} and Theorem \ref{thm_reg_1} or compute the projective dimension and regularity of the $3$-path ideal of $G_{z_1,2}$ and get $\pd (I_3(G_{z_1,2}) ) = 6$ and $\reg (I_3(G_{z_1,2})) = 7$. Hence, $\pd (I_3(G)) = 12$ and $\reg (I_3(G)) = 9$.
\end{exm}

\begin{rem} \begin{enumerate}
    \item Our method applies to study the projective dimension of edge ideal of unicyclic graphs and provides an alternative argument for the result of Alilooee, Beyarslan, and Selvaraja \cite{ABS}. We leave it as an exercise for interested readers.
    \item The formula for the projective dimension of $I_3(G)$ in Theorem \ref{thm_pd_1} and Theorem \ref{thm_pd_2} is more complicated than the formula for the regularity of $I_3(G)$ in Theorem \ref{thm_reg_1} and Theorem \ref{thm_reg_2}. As in Example \ref{ex1}, the term $p_3 + |N(x_0)| + s$ in Theorem \ref{thm_pd_1} cannot be dropped.
\end{enumerate}    
\end{rem}
We are now ready for the proof of Theorem \ref{thm_reg}.
\begin{proof}[Proof of Theorem \ref{thm_reg}]
    We prove by induction on $|G \setminus C_m|$. If $G = C_m$ the conclusion follows from \cite{AF1}. Thus, we may assume that there exists a leaf $z_0 \notin \{v_1,\ldots,v_m\}$ of some tree $T_i$ attached to $v_i$. We do not assume that $m>1$, i.e., $G$ might be a tree. We choose $z_0$ of the highest level in $T_i$. First, assume that $\level(z_0) \ge 2$. By Theorem \ref{thm_reg_1}, we have 
\begin{equation}\label{eq_4_1}
    \reg (I_3(G)) = \begin{cases}
    \max \{ \reg (I_3(G_{z_0,1})), \reg (I_3(G_{z_0,2}) ) +2\} & \text{ if } s \ge 1,\\
    \max \{ \reg (I_3(G_{z_0,1})), \reg (I_3(G_{z_0,3}) ) +2\} & \text{ if } s = 0.
\end{cases}
\end{equation}

If $\level(z_0) = 1$, by Theorem \ref{thm_reg_2}, we have 
\begin{equation}\label{eq_4_2}
    \reg (I_3(G)) = \begin{cases}
    \max \{ \reg (I_3(G_{z_0,1})), \reg (I_3(G_{z_0,2}) ) +2\} & \text{ if } s \ge 1,\\
    \max \{ \reg (I_3(G_{z_0,1})), \reg (I_3(G_{z_0,3}) ) +2, \reg (I_3(G_{z_0,4})) + 2\} & \text{ if } s = 0.
\end{cases}
\end{equation}

For ease of reading, we divide the proof into several steps. The first step is the result of Kumar and Sarkar \cite{KS}. We give a different proof here to illustrate our method.

\vspace{2mm}
\noindent \textbf{Step 1.} Assume that $G$ is a tree. Then $\reg (I_3(G)) = 2 \nu_3(G) + 1$. 

If $I_3(G) = 0$ then $\nu_3(G) = 0$ and the conclusion is vacuous. Now assume that $I_3(G)$ is nonzero. We can always choose a vertex $v$ acting as the root of $G$ so that $\level(z_0) \ge 2$. The conclusion then follows from induction, Eq. \eqref{eq_4_1}, and Lemma \ref{lem_mat1}.

\vspace{2mm}
\noindent \textbf{Step 2.} $\reg (I_3(G)) \le 2 \nu_3(G) + (\reg (I_3(C_m)) - 2\nu_3(C_m))$. 

Since the graphs $G_{z_0,j}$ for $j = 1,2,3,4$ are induced subgraphs of $G$, by induction and Step 1, we deduce that $\reg (I_3(G_{z_0,j})) \le 2 \nu_3(G) + (\reg (I_3(C_m)) - 2 \nu_3(C_m))$. If $\level(z_0) \ge 2$, the conclusion follows from Eq. \eqref{eq_4_1} and Lemma \ref{lem_mat1}. If $\level(z_0) =1$, the conclusion follows from Eq. \eqref{eq_4_2} and Lemma \ref{lem_mat2}.

\vspace{2mm}
\noindent \textbf{Step 3.} Assume that $G$ is unicyclic and $m = 3$ or $m \equiv 0,1 \pmod 4$. Then $\reg (I_3(G)) = 2 \nu_3(G) + 1$. 

Indeed, the conclusion follows from Lemma \ref{lem_induced_sub}, Step 2, and \cite[Corollary 5.5]{AF1}.

\vspace{2mm}
\noindent \textbf{Step 4.} Assume that $G$ is unicyclic with the unique cycle $C$ on $m$ vertices where $m > 3$. Assume further that $m \equiv 2,3 \pmod 4$ and $G$ is $3$-proximal. Then $\reg (I_3(G)) = 2 \nu_3(G) + \reg(I_3(C)) - 2 \nu_3(C)$. 

Indeed, we have $\Gamma_G(C)$ is the disjoint union of $C$ and a forest $F$. By Step 1, we deduce that $\reg (I_3(\Gamma_G(C))) = \reg (I_3(C))  + 2 \nu_3(F) + 1$. By Lemma \ref{lem_ind_mat_properties}, Lemma \ref{lem_induced_sub}, and Step 2, we have 
\begin{equation*}
    \begin{split}
        2 \nu_3(G) &+ (\reg (I_3(C_m)) - 2\nu_3(C_m)) = \reg (I_3(\Gamma_G(C))) \le \reg (I_3(G)) \\
        &\le 2 \nu_3(G) + (\reg (I_3(C_m)) - 2\nu_3(C_m)).
    \end{split}
\end{equation*} 
The conclusion follows.

\vspace{2mm}
\noindent \textbf{Step 5.} Assume that $G$ is unicyclic with the unique cycle $C$ on $m$ vertices where $m > 3$. Assume further that $m \equiv 2,3 \pmod 4$ and $G$ is not $3$-proximal. Then $\reg (I_3(G)) = 2 \nu_3(G) + 1$. 

By Lemma \ref{lem_induced_sub}, it suffices to prove that $\reg (I_3(G)) \le 2 \nu_3(G) + 1$ by induction on $|G \setminus C_m|$. First, assume that $\level(z_0) \ge 2$. The conclusion follows from induction, Eq. \eqref{eq_4_1}, and Lemma \ref{lem_3_mat_proximal_1}. Now, assume that $\level(z_0) = 1$. The conclusion follows from induction, Eq. \eqref{eq_4_2}, Step 1, and Lemma \ref{lem_3_mat_proximal_2}.

\vspace{2mm}
\noindent \textbf{Step 6.} Conclusion step. By Step 1 and Step 3, we may assume that $m > 3$ and $m \equiv 2,3 \pmod 4$. By Step 5, we may assume that $G$ is $3$-proximal. The conclusion then follows from Step 4 and \cite[Corollary 5.5]{AF1}.
\end{proof}

The following conjecture is a natural generalization of Theorem \ref{thm_reg} for $t$-path ideals.

\begin{conj}\label{conj_t}
    Assume that $G$ is a connected simple graph having at most one cycle $C$ of length $m$. Let $t \ge 2$ be a natural number. Then 
    $$\reg (S/I_t(G)) = \begin{cases} (t-1) \nu_t (G) + \left ( \reg (S/I_t(C)) - (t-1) \nu_3(C) \right ), & \text{ if } G \text { is } t\text{-proximal},  \\
(t-1) \nu_t(G), & \text{ otherwise}.    
\end{cases}$$
In particular, $\reg (S/I_t(G)) = (t-1) \nu_t(G)$ when $G$ is a tree.
\end{conj}

\begin{rem} \begin{enumerate}
    \item The case $t = 2$ of Conjecture \ref{conj_t} is \cite[Corollary 3.9]{ABS}. The case $t=3$ is Theorem \ref{thm_reg}.
    \item It is possible to establish that $z_0$ is a splitting vertex for $I_t(G)$ where $z_0$ is a leaf of the highest level of $G$. But finding the intersection corresponding to this splitting is more complicated when $t \ge 4$.
\end{enumerate}    
\end{rem}

\vspace{1mm}

\noindent {\bf Data Availability} Data sharing does not apply to this article as no datasets were generated or analyzed during the current study.

\vspace{1mm}

\noindent {\bf Conflict of interest} There are no competing interests of either a financial or personal nature.


\begin{thebibliography}{2}


\bibitem[ABS] {ABS} 
A. Alilooee and S. Beyarslan, S. Selvaraja, 
{\em Regularity of powers of edge ideals of unicyclic graphs}, Rocky Mountain J. Math.  {\bf 49} (2019), 699--728.

\bibitem[AF1]{AF1}
A. Alilooee, S. Faridi, 
{\em On the resolution of path ideals of cycles}, Commun. Algebra {\bf 43}, 5413--5433 (2015).


\bibitem[AF2]{AF2} 
A. Alilooee, S. Faridi, 
{\em Graded Betti numbers of path ideals of cycles and lines}, J. Algebra Appl. {\bf 17}, 1850011, (2018).




\bibitem[B]{B}	
A. Banerjee, 
{\it Regularity of path ideals of gap free graphs},
J. Pure Appl. Algebra {\bf 221} (2017), 2409--2419.


\bibitem[BB]{BB}
R. R. Bouchat and T. M. Brown,
\emph{Multi-graded Betti numbers of path ideals of trees},
J. Algebra Appl. {\bf 16} (2017), 1750018.




\bibitem[BHO]{BHO}
R. R. Bouchat, H. T. Ha, A. O'Keefe,
{\em Path ideals of rooted trees and their graded Betti numbers}, J.
Combin. Theory Ser. A {\bf 118} (2011), 2411--2425.


\bibitem[CD]{CD}
A. Conca, E. De Negri,
\emph{M-sequences, graph ideals, and ladder ideals of linear type},
J. Algebra {\bf 211} (1999), 599--624.


\bibitem[CF] {CF} 
M. Crupi, A. Ficarra, {\em Very well–covered graphs by Betti splittings}, J. Algebra {\bf 629} (2023), 76--108.


\bibitem[CFL]{CFL}
M. Crupi, A. Ficarra, E. Lax, 
{\em Matchings, squarefree powers and Betti splittings}, arXiv:2304.00255.

\bibitem [E]{E} 
N. Erey,
{\em Multigraded Betti numbers of some path ideals}, In: Combinatorial Structures in Algebra and Geometry, Springer Proc. Math. Stat. {\bf 331} (2020), 51--65.



\bibitem[FHV] {FHV} 
C. Francisco, H.T. Ha, and A. Van Tuyl, 
{\it Splittings of monomial ideals}, Proc. Amer. Math. Soc. {\bf 137} (2009), 3271--3282.


\bibitem[HV]{HV}
J. He and A. Van Tuyl,
{\em Algebraic properties of the path ideal of a tree},
Comm. Algebra {\bf 38}, (2010), 1725--1742.

\bibitem[J]{J}
S. Jacques,
{\em Betti numbers of graph ideals,} Ph.D. thesis, The University of Sheffield, arXiv:0410107 (2004).
 

\bibitem[KM] {KM} 
D. Kiani and S. S. Madani,
{\em Betti numbers of path ideals of trees},
Comm. Algebra {\bf 44} (2016), 5376--5394.


\bibitem[M2] {M2} 
D. R. Grayson and M. E. Stillman,
{\em Macaulay2, a software system for research in algebraic geometry,}
Available at: http://www.math.uiuc.edu/Macaulay2/.

 
\bibitem[KS]{KS}
R. Kumar and R. Sarkar,
{\em Regularity of 3-path ideals of trees and unicyclic graphs,} Bull. Malays. Math. Sci. Soc. {\bf 47} (2024).



\bibitem[NV1] {NV1} 
H. D. Nguyen and T. Vu,
{\em Linearity defect of edge ideals and Fr\"oberg’s theorem,} 
	J. Algbr. Comb. {\bf 44} (2016), 165--199.


\bibitem[NV2] {NV2} 
H. D. Nguyen and T. Vu,
{\em Powers of sums and their homological invariants,} 
	J. Pure Appl. Algebra {\bf 223} (2019), 3081--3111.

 
\bibitem[OHH]{OHH}
 H. Ohsugi, J. Herzog, and T. Hibi, \emph{Combinatorial pure subrings,} Osaka J. Math. {\bf 37} (2000), 745--757.



\end{thebibliography}
\end{document}